\documentclass[12pt]{amsart}
\usepackage[utf8]{inputenc}
\usepackage{amsmath}
\usepackage{amssymb}
\usepackage{indentfirst}
\usepackage[english]{babel}
\usepackage[style=alphabetic]{biblatex}
\usepackage[margin=1in]{geometry}
\usepackage[colorinlistoftodos]{todonotes}
\usepackage{amsthm}
\newtheorem{theorem}{Theorem}[section]
\newtheorem{corollary}[theorem]{Corollary}

\newtheorem{definition}[theorem]{Definition}

\title{Investigations of Impartial Games with a Pass}
\author{Emet Hirsch}
\address{Euler Circle, Palo Alto, CA 94306}
\email{emet.hirsch@gmail.com}
\date{\today}

\begin{document}

\begin{abstract}
    The study of the combinatorial game Nim and its variants is rich and varied, but little is known of the game Nim with a Pass. It is Nim, but once per game a player is permitted to skip their turn but this can only be done if a nonempty pile remains. In this paper we define a new binary operation on games which we use to prove novel properties of Nim with a Pass, as well as games which are generalizations of it. Most importantly, we make small progress on finding the $\mathcal{P}$-positions and optimal strategies of the game.
\end{abstract}

\maketitle

\section{Preliminaries and Fundamentals}
      Combinatorial game theory has always been exemplified by \textsc{Nim}, the first combinatorial game studied in the modern history of the field. Despite \textsc{Nim} having been solved over a century ago, variants of the game and extensions of it continue to interest mathematicians and game hobbyists alike.

Normally, combinatorial games such as \textsc{Nim} are played with two players alternating moves until one player has no legal moves, at which point that player is the loser.  \textsc{Nim} is a game consisting of a number of piles of stones. A turn consists of a player removing one or more stones from a single pile. We consider the game \textsc{Nim with a Pass}, i.e. \textsc{Nim} but there is a special move called \textit{pass} that can only be taken once (by either player) but may not be taken when there are no other legal moves. In the case of \textsc{Nim}, there are no legal moves when all piles are empty. The structure of games with passes is not well understood, and even the notoriously simple \textsc{Nim} becomes complex and dynamic when given a pass \cite{dynam}. General results about games with passes are few and far between. Existing research on \textsc{Nim with a Pass} focuses mostly on establishing patterns and completely characterizing the structure of small cases \cite{map,atlas}. Existing research on the dynamics of games with passes (and generalizations of such) in general has mostly investigated the problem from the perspective of the pass altering the traditional outcome function of a game, rather than altering the game itself. One example of this is the study of the pass as a type of perturbation in the outcome function on games, represented as a recurrence relation to permit the use of dynamical systems theory \cite{dynam}.

In this paper, we present a rigorous definition of a game with a pass, then extend the core idea of this definition to create a binary operation on games, called the \emph{split sum}. Then we will investigate and prove some of the properties of the split sum. The most important property that we will consider affects the losing positions of split sum games. Then, we investigate the applications of this result (and a corollary of it) in the study of \textsc{Nim with a Pass}.

In general, the notation used by this paper will mirror the common notation used in \cite{conway}. To denote a game of \textsc{Nim} with piles of sizes $a_1,a_2,\ldots, a_n$, we use the ordered tuple notation $(a_1,a_2,\ldots, a_n)$. To describe games in general, we use Conway's option notation. We are exclusively interested in \textit{impartial games}, i.e.\ games where each player has the same possible moves. So $\{g_1,g_2,\ldots,g_n\}$ simply denotes the impartial game $G$ with options $\{g_1,g_2,\ldots,g_n\}$. The symbol + refers to the disjunctive sum of games, as is standard notation. We write $g$ for an option of $G$, and $g'$ for an option of an option of G. We write $g \in G$ to say that $g$ is an option of $G$. A game's \textit{birthday} is the height of its game tree or, equivalently, its position in the Conway hierarchy of games. A game of normal \textsc{Nim} has birthday equal to its number of stones. We occasionally write $b(G)$ for the birthday of game $G$. A game's \textit{Grundy value} refers to the unique one-pile game of \textsc{Nim} to which it is equal, due to the Sprague--Grundy Theorem \cite{conway}. Similarly, a \textit{nimber} is any game of one-pile \textsc{Nim}, and we refer to a nimber by $*n$. It should be noted that every game of \textsc{Nim} is simply a disjoint sum of nimbers. Sometimes, we will have to refer to an option of a game denoted by a lowercase letter or a nimber; we will also use $'$ in this case. For example, an option of $*n$ would be written $*n'$.

Due to the behaviour of games with passes, this paper will have a special notational concern. Normally in combinatorial game theory, we say two games $G$ and $H$ are \textit{equal} for when all games $X$, we have $o(G+X)=o(H+X)$. $o(G)$ for some game $G$ simply refers to the status of $G$ as a $\mathcal{P}$-position or $\mathcal{N}$-position. So this definition states that for all other games $X$, the same player wins the disjoint sum of $G$ and $X$ and the disjoint sum of $H$ and $X$. We continue this notation. However, we have to be careful to distinguish equality from the stronger claim that games $G$ and $H$ are equivalent. We say two games are \textit{equivalent} if the have exactly the same set of options. That is to say, $G$ and $H$ are equivalent if for every element in $G$, there is exactly one equivalent element in $H$ and vice versa. This distinction is necessary because giving a game the pass is not well-defined with respect to the standard equivalence relation on the class of games. When two games $G$ and $H$ are equivalent we write $G \equiv H$. In standard notation, 0 refers to any $\mathcal{P}$-position. That is to say, any $\mathcal{P}$-position is often written as 0 even when that $\mathcal{P}$-position has options. We write $G=0$ when $G$ is a $\mathcal{P}$-position as normal, but we write $\mathbf{E}$, standing for ``empty," to distinguish the game with no options.

Firstly, we must capture the nature of the pass with a rigorous definition using Conway's option notation. We can take a game $G$ and consider $G$ \textit{with a pass}, a game where the following moves are legal:

\begin{enumerate} \item A move in $G$, which reduces it to one of its options while the pass remains.

\item \textit{Taking the pass}, which reduces the game $G$ with a pass to simply $G$. This move is not permitted when $G$ has no options left. \end{enumerate}

We write $G$ with a pass as $G^\ast$. Understanding the intuitive definition of games with a pass, we now seek to rigorously define a function that takes a game $G$ to the corresponding game with a pass.
\begin{definition}
\label{passdefinition}
\[G^\ast \equiv \begin{cases} 
      \mathbf{E} & \text{when } G\equiv \mathbf{E} .\\
      \{G, g^\ast\} \text{ where $g$ ranges over the options of $G$} & \text{otherwise.}
   \end{cases}
\]
\end{definition}

The definition takes \textbf{E} to itself because the pass cannot be used when there the game has no remaining moves. It takes all other games $G$ to a game with options of two types:

\begin{enumerate} \item \textit{Passing}, which goes to $G$.

\item Make a move in $G$, which goes to an option of $G$ that has the pass recursively applied to it because the pass can still be taken. \end{enumerate}

Thus, this definition captures the notion of giving a game the pass.

Now is a good time to justify our peculiar notation. Observe that the pass is not well-defined with respect to the normal definition of equality of games. As an example, consider the following. The game with no options is a $\mathcal{P}$-position, and continues to be a $\mathcal{P}$-position when given a pass. The game of \textsc{Nim} with two piles of size one is a $\mathcal{P}$-position, but that game with a pass is not, and is in fact equal to the game of \textsc{Nim} with one pile of size one. We might write this as follows: \[\mathbf{E}=0 \qquad\text{and} \qquad (1,1)=0.\]
However, \[\mathbf{E}^*=0 \qquad \text{and} \qquad (1,1)^*=(1).\]

As we have previously stated, $(1,1)$ refers to the game of \textsc{Nim} with two piles of size one. The values used as an example can be verified by computation.

\section{Split Sums}

One of the reasons that games with a pass are of interest is because giving a game a pass is very similar to simply considering its disjoint sum with $*1$, the game of \textsc{Nim} with one pile of size one. However, its initial value at \textbf{E} is perturbed and, since the operation is defined recursively, the behaviour of the operation becomes very complex. We call this operation the \textit{split sum} and denote the split sum of $G$ and $H$ by $G \circ H$. In the following definition, $g$ ranges over the options of $G$ and $h$ ranges over the options of $H$.

\begin{definition} 
\label{splitdefinition}
\[G \circ H =
 \begin{cases} 
      \mathbf{E} & \text{when } G \equiv \mathbf{E}, \\
      $G$ & \text{when } H \equiv \textbf{E}, \\
      
      \{G \circ h, g \circ H\}  & \text{otherwise.}
   \end{cases}
\]
\end{definition}

Note that this definition is nearly exactly as the same as the definition for addition (disjoint sum) of games. The only difference is that, when $G \equiv \textbf{E}$, the definition mandates that the output be \textbf{E}. In normal addition, the recursive part at the bottom would apply and the output would be $H$.

To intuitively understand split sums, consider the following description of the pass function of games. ``A game G with a pass is like $G+*1$, but a player is only allowed to play in the $*1$ component if there are still legal moves in the G component." The intuitive nature of $G \circ H$ is simply $G+H$, but players are only allowed to play moves in $H$ if there are still legal moves in $G$.

Before we begin to prove properties of split sums, we should note: $G \circ *1 \equiv G^*$. Split sums are motivated by a desire to generalize the notion of games with a pass. There are other classes of games or functions on games that seem to generalize the concept of a pass, but they are often extremely difficult to work with and compute. While this is still true for split sums, they have some surprisingly symmetrical properties that mirror those of normal disjunctive sums.

 To begin, let us note that split sums are neither commutative nor associative. The first of the following two pairs of equations demonstrate that split sums are not commutative, while the second demonstrates that they are not associative. They can be verified by computation.
\[
    *1 \circ \textbf{E} \equiv *1, \text{ but } \textbf{E} \circ *1 \equiv \textbf{E} .
\]
\[
    (*1 \circ \textbf{E}) \circ *1 \equiv *2, \text{ but } *1 \circ ( \textbf{E} \circ *1) \equiv *1.
\]

\begin{theorem}
\label{Theorem 2.2}
If G and H are $\mathcal{P}$-positions, then $G \circ H$ is also a $\mathcal{P}$-position.
\end{theorem}

\begin{proof} As base cases, we know by definition that
\begin{gather*}
    G \circ \textbf{E} \equiv G \text{ and } \textbf{E} \circ H \equiv \textbf{E}.
\end{gather*}

Now, consider the $\mathcal{P}$-positions $G$ and $H$ such that their split sum is an $\mathcal{N}$-position and the sum of their birthdays is minimal among all pairs of games with that property. (That is, $G \circ H$ is an $\mathcal{N}$-position and $b(G)+b(H)$ is minimal.)
By our base case, both of those games must have options. Since the game is an $\mathcal{N}$-position, Player 1 must have a winning move. If that winning move is in $H$, then it goes to $h$. Since $H$ (by assumption) is a $\mathcal{P}$-position, $h$ is an $\mathcal{N}$-position, and Player 2 can now move it to some $\mathcal{P}$-position $h'$. But now Player 1 is left with the game $G \circ h'$, which must be a $\mathcal{P}$-position because $G$ and $h'$ are $\mathcal{P}$-positions and $h'$ has strictly less birthday than $H$, meaning our theorem applies to it.

Thus, we can reduce to the case where Player 1's winning move is in $G$. If that winning move is in $G$, then it goes to $g$. Since $G$ (by assumption) is a $\mathcal{P}$-position, $g$ is an $\mathcal{N}$-position, and Player 2 can now move it to some $\mathcal{P}$-position $g'$. But now Player 1 is left with the game $g' \circ H$, which must be a $\mathcal{P}$-position because $G$ and $h'$ are $\mathcal{P}$-positions and $g'$ has strictly lower birthday than $G$, meaning our theorem applies to it. But that means Player 1 does not have a winning strategy, which is a contradiction.
\end{proof}

\begin{corollary}
If $G$ is a $\mathcal{P}$-position and $H$ is an $\mathcal{N}$-position (or vice versa), $G \circ H$ is an $\mathcal{N}$-position.
\end{corollary}

\begin{proof} It suffices to show a move to a $\mathcal{P}$-position. Whichever element of the sum is an $\mathcal{N}$-position has some move to a $\mathcal{P}$-position. If player takes this move, the resulting game is a split sum of two $\mathcal{P}$-positions, and therefore a $\mathcal{P}$-position by Theorem \ref{Theorem 2.2}. Thus, this move is a winning move for Player 1, therefore the game is an $\mathcal{N}$-position.
\end{proof}

\begin{theorem}
\label{strangeassociativity}
Let $G$, $H$, and $K$ be games. Then 
    \[(G  \circ H)  \circ K = G \circ (H+K).\]
\end{theorem}

\begin{proof} First, let us establish our base case. Let all three games be \textbf{E}. In this case, the equality holds.

Now, consider the games $G$, $H$, and $K$ that violate the equality such that $b(G)+b(H)+b(K)$ is minimal. Two games being equal is equivalent to their sum being a $\mathcal{P}$-position; thus it suffices to prove 
 \begin{gather*}
    (G  \circ H)  \circ K + G \circ (H+K) = 0.
\end{gather*}

For this to be false, it must be an $\mathcal{N}$-position instead. Thus, Player 1 must have a winning strategy. That winning strategy might be in the first part of the sum, or in the second part. Player 2 can simply make the corresponding move in the opposite part of the sum, leaving a game of similar form. But this game that Player 2 left must be a $\mathcal{P}$-position due to its lower birthday-sum, a contradiction.
\end{proof}
\begin{theorem}
Let $G$, $H$, and $K$ be games and let $H=K$. Then $G \circ H = G \circ K$.
\end{theorem}

\begin{proof} Assume, for the sake of contradiction, that there are three games such that this does not hold. Then let us choose $G$, $H$, and $K$ such that the sum of the birthdays of those games is minimal among games that do not follow our theorem. It suffices to show that $G \circ H + G \circ K = 0$. We know by assumption that $H+K=0$. 

\textit{Case 1:} Imagine Player 1 makes a move in $G \circ H$ to $g \circ H$ or in $G \circ K$ to $g \circ K$. Then Player 2 can make the same move in the other $G$ component, leaving Player 1 with $g \circ H + g \circ K$ which is a $\mathcal{P}$-position by assumption. Thus Player 1 has no winning strategy in a $G$ component.

\textit{Case 2:} Imagine Player 1 makes a move in $G \circ H$ to $G \circ h$. Without loss of generality, this reasoning also applies if they make the same type of move in the $K$ component. Since $H+K$ is a $\mathcal{P}$-position, Player 2 has some countermove to $h'+K$ or $h+k$ which are both $\mathcal{P}$-positions and therefore $h'=K$ or $h=k$, respectively. Player 2 can employ this countermove, leaving Player 1 with $G \circ h' + G \circ K$ or $G \circ h + G \circ k$. In either case, by assumption, this is a $\mathcal{P}$-position and thus Player 1 has no winning moves in the $H$ or $K$ component. Thus Player 1 has no winning moves at all, a contradiction.
\end{proof}

\section{Results}

Let us make a preliminary note about the nimbers. Observe that if $G$ is a nimber, then for all $g', g'\in g \in G $ implies that $ g' \in G$. We call this property \textit{hereditary transitivity}.

 Also note that every option of a nimber is a nimber, so this rule recurs. We now have all the tools necessary to prove the major result, the Nimber Extension Rule.

\begin{theorem}[Nimber Extension Rule] Let $G$ be any impartial game and $*n$ be a nimber. Let $H$ be any impartial game.
Then $G \circ H +*n=0 \iff (G+*n) \circ H=0$.
\end{theorem}

\begin{proof} We prove the left-right implication first. To do so, we assume a contradiction of smallest birthday-sum. Then, the right-left implication follows quickly from the left-right implication. 

First, let us prove the left-right implication. We want to show the left-right implication, that is:
\begin{gather*}
    G \circ H +*n=0 \text{ implies that } (G+*n) \circ H=0. 
\end{gather*}

First, consider the base case where $G \equiv \textbf{E}$. Then $ *n \equiv \textbf{E}$ and the implication can be seen from the fact that both its expressions are simply $ G \circ H$.

Consider the base case where $H \equiv \textbf{E}$. Since  for all $G$,$ \textbf{ } G \circ \textbf{E} \equiv G $ by definition, the implication can be seen from the fact that both its expressions are simply $ G+*n$.

Observe that if $G \circ H$ is a $\mathcal{P}$-position, $*n \equiv \textbf{E}$. Then the implication can be seen in the same way as when $ G \equiv \textbf{E}$.

Having finished that first step, let us rewrite the desired statement as ``If $G \circ H +*n$ is a $\mathcal{P}$-position, then $(G+*n) \circ H$ is a $\mathcal{P}$-position." A game being a $\mathcal{P}$-position is equivalent to all its options being $\mathcal{N}$-positions, so we rewrite again, subdividing the options of each game into three categories which encompass all the games options.

Now, we claim that if all options of the forms $G \circ h +*n$, $g \circ H+*n$, and $G \circ H +*n'$ are all $\mathcal{N}$-positions, then all options of the forms $ ( G + *n ) \circ h $, $(g + *n) \circ H$, and $(G+*n') \circ H $ are $\mathcal{N}$-positions. As we explained in the previous paragraph, proving this statement suffices to show the desired result.

Note the correspondence between the three cases given to be $\mathcal{N}$-positions with the three desired cases. We prove them in the order that they are written. Assume, for the sake of contradiction, that the theorem does not hold. We choose $G$ and $H$ such that their birthday is minimal among pairs of games such that the theorem does not hold. By our base cases, both have at least one option and $*n \neq \textbf{E}$.

\textit{Case 1:} We already are given that $G \circ h+*n$ is an $\mathcal{N}$-position, and we want to show that $(G+*n) \circ h$ is an $\mathcal{N}$-position. The first game, by virtue of being an $\mathcal{N}$-position, has at least one $\mathcal{P}$-position as an option. That option will be of the form $g \circ h +*n$, $G \circ h' +*n$, or $G \circ h +*n'$. To show that the second game is an $\mathcal{N}$-position, it suffices to show that it has one $\mathcal{P}$-position as an option, which may be of the form $(g+*n) \circ h$, $(G+*n) \circ h'$, or $(G+*n') \circ h$.

Of three games that are options of the guaranteed $\mathcal{N}$-position, they all have lesser birthday-sum than $G$ and $H$, thus our theorem holds on them. Therefore, whichever one is a $\mathcal{P}$-position, it implies one of latter three games is a $\mathcal{P}$-position by the theorem. This suffices to show that the first category of the options of $G+*n \circ H$ consists of $\mathcal{N}$-positions, the first case in proving the left-right implication.

\textit{Case 2:} This is very similar to Case 1. We are given that $g \circ H +*n$ is an $\mathcal{N}$-position and we want to prove that $(g +*n) \circ H $ is an $\mathcal{N}$-position. Now, we rewrite the statements the same way as we did in Case 1. We now know that at least option of the form $g' \circ H + *n$, $g \circ h + *n$, or $g \circ H + *n'$ is a $\mathcal{P}$-position. It suffices to show that at least one of $g'+*n \circ H$, $g + *n \circ h$, or $g + *n' \circ H$.

Of three games that are options of the guaranteed $\mathcal{N}$-position, they all have lesser birthday-sum than $G$ and $H$, thus our theorem holds on them. Therefore, whichever one is a $\mathcal{P}$-position, it implies one of latter three games is a $\mathcal{P}$-position by the theorem. This suffices to show that the second category of the options of $G+*n \circ H$ consists of $\mathcal{N}$-positions, the second case in proving the left-right implication. Note that this logic is identical to that of Case 1.

\textit{Case 3:} Now, to prove the left-right implication of the theorem, it suffices to show that all options of the form $(G+*n') \circ H $ are $\mathcal{N}$-positions. We are given the fact that all options of the form $G \circ H+*n'$ are $\mathcal{N}$-positions. Using the same method as before, we break each of those statements into categories and rewrite them as statements about $\mathcal{P}$-positions.

We can now claim that at least one of the games of the form $G \circ h +*n'$, $g \circ H +*n'$, or $G \circ H+*n``$ is a $\mathcal{P}$-position. We want to show that at least one of the games of the form $G \circ h+*n'$, $(g+*n') \circ H$, or  $(G+*n``) \circ H$ is a $\mathcal{P}$-position.

The third subcase of the given actually can never happen. As we mentioned earlier, an option of an option of a nimber is also an option of that nimber, so any option of the form $G \circ H + *n``$ is also of the form $G \circ H + *n'$ which we know to be an $\mathcal{N}$-position because it is an option of $G \circ H +*n$ which is a $\mathcal{P}$-position by assumption.

Thus, we now have that at least one of the games of the form $G \circ h +*n'$ or $g \circ H +*n'$ is a $\mathcal{P}$-position. We want to show that at least one of the games of the form $(G+*n') \circ h$, $(g+*n') \circ H$, or $(G+*n``) \circ H$ is a $\mathcal{P}$-position.

It now suffices to show that at least one of the games of the form $G+*n' \circ h$ or $(g+*n') \circ H$ is a $\mathcal{P}$-position. We can use the same logic as the previous two cases to show this. This completes Case 3, and thereby the left-right implication of the equality. 

Now that we are done with the left-right implication of the equality, we can use it to prove the right-left implication of the equality. We write the right-left implication as follows.
\begin{gather*}
    G \circ H +*n =0 \text{ is implied by } (G+*n) \circ H=0. 
\end{gather*}

Assume, for the sake of contradiction, that $(G+*n) \circ H=0 $  but $G \circ H+*n \neq0$. $G \circ H $ is an impartial game, so by the Sprague-Grundy Theorem \cite{conway} there exists $ *m $ such that $ G \circ H+*m=0. $ If $ *m=*n, $ contradiction.

 If $*m \neq *n$, then we have $ G \circ H+*m=0$ and therefore by the left-right implication $ (G+*m) \circ H=0$. We also have $ (G+*n) \circ H=0$. Since $ *m \neq n$, one of those two games is an option of the other, despite the fact that they are both $\mathcal{P}$-positions. This is a contradiction. 
\end{proof}

Now, let us clarify some language and use it to prove a corollary.

Let $G$ be a game of \textsc{Nim} and $H$ an impartial game, such that $G \circ H $ is an $\mathcal{N}$-position. $G$ is a game of \textsc{Nim} and is thereby composed of piles, so we can write out $G$ as $(a, b, \ldots)$. We assume that $ G \circ H $ has two different \textit{winning moves} (moves to $\mathcal{P}$-positions) to
$ (a', b, \ldots) \circ H $ and $ (a, b', \ldots) \circ H. $ Effectively, we are assuming that there are at least two piles from which a player can remove stones to create a $\mathcal{P}$-position.  Note that since only one pile gets changed in a move, the game's other piles which we omit with an ellipsis are identical. Having defined these terms, we state that

\begin{corollary}
\label{sumrule}
    $*a'+*a=*b'+*b.$
\end{corollary}
That is to say, the nim-sum of $a$ and the winning move from $a$ will be the same as the nim-sum of $b$ and the winning move from $b$.

\begin{proof} Consider the following game.
\begin{gather*}
    (a', b, \ldots) \circ H +(a, b', \ldots) \circ H + *a'+*a+*b'+*b
\end{gather*}

Observe that that this game is completely symmetrical. If Player 1 makes a move in any component, Player 2 can make the same move in the corresponding component, leaving another game of this form. When all components are E, the game is a $\mathcal{P}$-position. There are exactly 2 of every individual component, and this information is sufficient to declare this game a $\mathcal{P}$-position. This is an informal writing of a standard ``tweedle-dee, tweedle-dum" symmetry argument. 

By assumption, we know $(a', b, \ldots) \circ \textbf{ } H $ and $ (a, b', \ldots) \circ H $ are both $\mathcal{P}$-positions, so their sum is also a $\mathcal{P}$-position. Thus we can write,

\begin{gather*}
        (a', b, \ldots) \circ H +(a, b', \ldots) \circ H + *a'+*a+*b'+*b = 0.
\end{gather*}
\begin{gather*}
            (a', b, \ldots) \circ H + (a, b', \ldots) \circ H = 0.
\end{gather*}

Subtracting the second equation from the first yields $*a'+*a+*b'+*b=0$, which is equivalent to the desired result.
\end{proof}

Before we proved the Nimber Extension Rule, we noted the property of Nimbers that we called hereditary transitivity that later ended up being surprisingly pivotal in the proof. This property is very rare among games, as we will soon show.

We say that a game is \textit{nimber-like} if the game has hereditary transitivity and all of its suboptions are also nimber-like. (That is to say, a game is nimber-like if it has hereditary transitivity and all its options have hereditary transitivity and all their suboptions have hereditary transitivity, and all their suboptions have hereditary transitivity, and so on.)

\begin{theorem} If a game $G$ is nimber-like and finite, then $G\equiv *n$ for some nimber $*n$. That is to say, $G$ is a nimber.
\end{theorem}

\begin{proof} Assume the theorem is false. Then there is some $G$ of minimal birthday that is nimber-like but is not a nimber. All of its options are nimber-like (by definition of the term) and therefore nimbers (by assumption.) Let $*k$ be the nimber of largest magnitude that is an option of $G$. Since $*k$ is a nimber, it contains all nimbers smaller than it. Since $G$ has hereditary transitivity (because it is nimber-like), it contains all of the options of $*k$, and therefore all of the nimbers up to $*k$. These are all of the elements of $G$, and therefore $G \equiv $ the nimber of magnitude $k+1$.
\end{proof}

This result shows that our proof of the Nimber Extension Rule can not be extended to cover games other than nimbers, since its necessary property is sufficient to characterize the nimbers if the property holds inductively. It is worth noting that this result can be extended to infinite nimbers by simply replacing ``the largest nimber" with an infinite ascending sequence of nimbers.

 Now that we have proven our results in the language of split sums, we can use them to investigate the problem that motivated them in the first place. We now show a novel result in the study of \textsc{Nim with a Pass} as well as novel proofs of older results.

The fundamental questions about \textsc{Nim with a Pass} are the same as those for any other impartial game: quickly determine whether an arbitrary position is an $\mathcal{N}$-position or a $\mathcal{P}$-position, quickly determine the winning move from an arbitrary $\mathcal{N}$-position, and quickly determine the Grundy value of an arbitrary position. The third task is usually the most difficult, since it obviously encompasses the first task and a solution to it usually quickly creates a solution to the second.

The first two tasks are completely solved when the underlying game of \textsc{Nim} only has two piles \cite{atlas} or has exclusively piles of size four or below \cite{map}. None of the tasks are solved in games of three or more piles with at least one pile of size greater than four.

The Nimber Extension Rule is the most important result in this paper because it makes slight progress on the solution of \textsc{Nim with a Pass} with three piles.

Imagine we have a game of two-pile \textsc{Nim with a Pass} $(a,b)^*$. It can also be written $(a,b) \circ *1$ but the former notation is more compact. By the Sprague-Grundy theorem, $(a,b)^*$ has some Grundy value $c$ which is a nimber. Then,
\[
    (a,b)^*+*c=0.
\]

By the Nimber Extension Rule, we have
\[
        (a,b)^*+*c=0 \iff (a,b,c)^*=0.
\]

Therefore,
\[
    (a,b,c)^*=0.
\]

Thus, finding the Grundy values of two-pile \textsc{Nim with a Pass} suffices to find the $\mathcal{P}$-positions of three-pile \textsc{Nim with a Pass} (and vice versa). 

This is progress on finding the $\mathcal{P}$-positions of \textsc{Nim with a Pass} with three or more piles, since finding the Grundy vales of the two-pile game may be easier. Additionally, if someone were to find a way to quickly generate the Grundy values of \textsc{Nim with a Pass} positions from its $\mathcal{P}$-positions, this could be converted into a fast algorithm to solve \textsc{Nim with a Pass} with an arbitrary number of piles. However, this task is very difficult for games in general and is likely infeasible.

Corollary \ref{sumrule} may also be useful in a clever algorithm designed to solve \textsc{Nim with a Pass}. Given a known $\mathcal{N}$-position and a known winning move from it, the corollary constrains all possible other winning moves
from it. If other results about the structure of \textsc{Nim with a Pass} are found, this could be used in combination with them to find $\mathcal{P}$-positions.

The analysis of split sums also offers alternate proofs of many existing results in the study of \textsc{Nim with a Pass}. For example, it is well known that that the only (nontrivial) $\mathcal{P}$-positions of two-pile \textsc{Nim} with the Pass are the pairs of the form $(a, a+1)^*$ for odd $a$. This is normally proved by a long induction on cases. We present a proof sketch of a shorter but still tedious alternate proof using the Nimber Extension Rule.

Firstly, one proves by a quick induction that the Grundy value of a game of \textsc{Nim with a Pass} consisting of a single pile of odd size is the pile's size plus 1, and similarly the value of an even pile is that pile's size minus 1. Then, one uses the Nimber Extension Rule on that (the same way we did a few paragraphs ago) to deduce that $(a, a+1)^*$ is a $\mathcal{P}$-position for odd $a$. Showing that those (and $(0,0)$) are the only $\mathcal{P}$-positions just takes some easy casework showing that all positions of other forms are $\mathcal{N}$-positions.

It is also well known that $(G^*)^*=G$. That is to say, $G$ with a pass with a pass equals $G$. This is normally proved by a standard symmetry argument, but the use of split sums offers an alternate proof that more efficiently encapsulates the same reasoning.

$(G^*)^* = (G \circ *1) \circ *1$ by Definition \ref{passdefinition} and Definition \ref{splitdefinition}.

$(G \circ *1) \circ *1 = G \circ (*1 + *1)$ by Theorem \ref{Theorem 2.2}.

$G \circ (*1 + *1)= G \circ \textbf{E} = G$ by Theorem \ref{strangeassociativity} and then Definition \ref{splitdefinition}.

While this paper, motivated by \textsc{Nim with a Pass}, developed the idea of the split sum, there are other ways to generalize \textsc{Nim with a Pass}. It is likely that many other fruitful avenues of research on the problem exist. The author believes a fast solution to \textsc{Nim with a Pass} exists based on an as of yet not understood structure in the problem because the Nimber Extension Rule (and its corollary) show that split sums obey very weak constraints similar to those obeyed by traditional \textsc{Nim}, and cleverer analysis may yield more powerful results or more powerful applications of existing ones.

\section{Acknowledgements}
The author would like to thank Simon Rubinstein-Salzedo  for introducing him to the problem, teaching him the necessary tools to investigate it, and assisting him with this paper's formatting and style. He would also like to thank Rachana Madhukara for helpful comments.

\printbibliography

\end{document}